\documentclass{amsart}
\usepackage{amsmath,amssymb,amsxtra,amsfonts,amsthm,comment,url}

\usepackage{tikz} 
\usetikzlibrary{arrows,backgrounds,decorations,automata}

\theoremstyle{plain}
\newtheorem{theorem}{Theorem}
\newtheorem{lemma}{Lemma}
\newtheorem{corollary}{Corollary}

\newtheorem{question}{Question}

\theoremstyle{definition}
\newtheorem{remark}{Remark}
\newtheorem{definition}{Definition}

\newcommand{\B}{\mathbb}

\newcommand{\ga}{\alpha}

\newcommand{\eps}{\varepsilon}

\newcommand{\gq}{\vartheta}

\begin{document}

\title{Strong normality and generalised Copeland--Erd\H{o}s numbers}

\author{Elliot Catt}
\address{School of Math.~and Phys.~Sciences\\
University of Newcastle\\
Callaghan\\
Australia}
\email{Elliot.Catt@uon.edu.au}

\author{Michael Coons}
\address{School of Math.~and Phys.~Sciences\\
University of Newcastle\\
Callaghan\\
Australia}
\email{Michael.Coons@newcastle.edu.au}

\author{Jordan Velich}
\address{School of Math.~and Phys.~Sciences\\
University of Newcastle\\
Callaghan\\
Australia}
\email{Jordan.Velich@uon.edu.au}

\thanks{The research of E.~Catt and J.~Velich was supported by CARMA Undergraduate Summer Research Scholarships, and the research of M.~Coons was supported by ARC grant DE140100223.}

\date{\today}


\begin{abstract} 
We prove that an infinite class of Copeland-Erd\H{o}s numbers are not strongly normal and provide the analogous result for Bugeaud's Mahler-inspired extension of the Copeland-Erd\H{o}s numbers. After the presentation of our results, we offer several open questions concerning normality and strong normality.
\end{abstract}

\maketitle

\section{Introduction}

Up to the year 1909, problems of probability were classified as either `discontinuous' or `continuous' (also called `geometric'). Towards filling this gap, in that year, Borel \cite{B1909} introduced what he called countable probabilities (probabit\'es d\'enombrables). In this new type of problem, one asks probabilistic questions about countable sets. As a (now very common) canonical example, Borel \cite{B1909} considered properties of the frequency of digits in the digital expansions of real numbers. Specifically, he defined the notions of {\em simple normality} and {\em normality}. 

A real number $x$ is called {\em simply normal to the base $b$} (or {\em $b$-simply normal}) if each of $0,1,\ldots,b-1$ occur in the base-$b$ expansion of $x$ with equal frequency $1/b$. This number $x$ is then called {\em normal to the base $b$} (or {\em $b$-normal}) provided it is $b^m$-simply normal for all positive integers $m$.

As part of his seminal work in the area, Borel \cite{B1909} showed that {\em almost all numbers are normal to every positive integer base}, but it was not until 1933 that Champernowne \cite{C1933} gave an explicit example; he showed that the number $$\xi_{\B{N},10}:=0.1234567891011121314151617\cdots,$$ produced by concatenating the digits of the positive integers, is normal to the base $10$. He also claimed that the number $$0.46891012141516182021222425\cdots,$$ produced by concatenating the digits of the composite integers, is normal to the base $10$, though he did not provide a proof. Such a proof was provided by Copeland and Erd\H{o}s \cite{CE1946}, who showed that if $b\geqslant 2$ is an integer and $a_1,a_2,a_3,\ldots$ is an increasing sequence of positive integers such that for every $\gq<1$ the number of $a_n$'s up to $N$ exceeds $N^\gq$ for $N$ sufficiently large, then the number $$\xi_{A,b}:=0.(a_1)_b(a_2)_b(a_3)_b\cdots$$ is $b$-normal, where $(a_n)_b$ denotes the base-$b$ expansion of the integer $a_n$ and $A:=\{a_1,a_2,a_3,\ldots\}.$

In the present day, normality is a property that one presumes a `random' number must have. But, of course, just because a number is normal to some base, does not mean this number is in any way `random'. Champernowne's number $\xi_{\B{N},10}$ is by no means `random'. In order to more formally (and quite simply) differentiate normal from `random', Belshaw and Borwein \cite{BB2013} introduced a new asymptotic test of pseudorandomness based on the law of the iterated logarithm, which they call {\em strong normality}. Specifically, let $\xi\in(0,1)$ and $m_{k,b}(\xi;n)$ denote the number of times that $k$ appears in the first $n$ $b$-ary digits of the base-$b$ expansion of $\xi$. The number $\xi$ is called {\em simply strongly normal in the base $b$} (or {\em $b$-simply strongly normal}) provided both $$\limsup_{n\to\infty} \frac{m_{k,b}(\xi;n)-n/b}{\sqrt{2n\log\log n}}=\frac{\sqrt{b-1}}{b}\quad\mbox{and}\quad\liminf_{n\to\infty} \frac{m_{k,b}(\xi;n)-n/b}{\sqrt{2n\log\log n}}=-\frac{\sqrt{b-1}}{b},$$ for every $k\in\{0,1,\ldots,b-1\}$. This number $\xi$ is then called {\em strongly normal to the base $b$} (or {\em $b$-strongly normal}) if it is $b^m$-simply strongly normal for all positive integers $m$. 

To illustrate the value of their new asymptotic test of pseudorandomness, Belshaw and Borwein \cite{BB2013} showed that a strongly normal (resp.~simply strongly normal) number was also normal (resp.~simply normal). They also showed that Champernowne's number $\xi_{\B{N},2}$, produced by concatenating the positive integers written in binary, is not strongly normal to the base $2$. They commented that this was true for every base $b$, and a slight modification of their proof indeed gives this result. 

In this paper, we show that an infinite class of Copeland-Erd\H{o}s numbers is not strongly normal. In fact, our result addresses Bugeaud's Mahler-inspired generalisation of the Copeland-Erd\H{o}s numbers. Recall that for a real number $x$, $\lfloor x\rfloor$ denotes the greatest integer which is less than or equal to $x$, and $\{x\}$ denotes the fractional part of $x$; that is, $x-\{x\}=\lfloor x\rfloor.$ To ease the exposition, before stating our result, we introduce a piece of notation in the following definition.

\begin{definition}\label{xiAbc} Let $c\geqslant 1$ be a real number, $b\geqslant 2$ be an integer, and $a_1,a_2,a_3,\ldots$ be an increasing sequence of positive integers. For each $n\geqslant 1$, let $(a_n)_b$ denote the base-$b$ expansion of the integer $a_n$. We define the real number $$\xi_{A,b,c}:=0.(a_1)_b\cdots(a_1)_b(a_2)_b\cdots(a_2)_b(a_3)_b\cdots(a_3)_b\cdots,$$ where $A:=\{a_1,a_2,a_3,\ldots\}$ and each block of $b$-ary digits $(a_n)_b$ is repeated $\lfloor c^{\ell_b(a_n)}\rfloor$ times; here $\ell_b(a_n)$ is equal to the length of the integer $a_n$ written in base $b$.
\end{definition}

We prove the following theorem.

\begin{theorem}\label{upAc} Let $c\geqslant 1$ be a real number, $b\geqslant 2$ be an integer, $A\subseteq\B{N}$, $A(x)$ be the number of elements in $A$ that are at most $x$, and $\ga$ be a real number satisfying $$\ga<\left(1-\frac{1}{b}+\frac{1}{b(bc-1)}\right)
\cdot \frac{\log b}{2}.$$ If for large enough $x$ $$A(x)\leqslant \ga\cdot\frac{x}{\log x},$$ then $\xi_{\B{N}\backslash A,b,c}$ is not $b$-simply strongly normal. 
\end{theorem}

Theorem \ref{upAc} contains the result of Belshaw and Borwein \cite{BB2013} that each Champernowne number $\xi_{\B{N},b}=\xi_{\B{N},b,1}$ is not $b$-strongly normal. It also yields the following corollary.

\begin{corollary} The number $$0.46891012141516182021222425\cdots,$$ produced by concatenating the digits of the composite integers, is not strongly normal to the base $10$.
\end{corollary}

\begin{remark} Bugeaud, in his monograph, {\em Distribution modulo one and Diophantine approximation} \cite[p.~87, Theorem 4.10]{B2012}, was the first to prove the $b$-normality of the generalised Copeland-Erd\H{o}s numbers (with the multiplicities as given in Definition~\ref{xiAbc}); the special case of the $b$-normality of $\xi_{\B{N},b}$ was proved by Mahler \cite{M1937}.
\end{remark}

\section{Non-strong normality and generalised Copeland-Erd\H{o}s numbers}

In this section, we prove Theorem \ref{upAc}. Our proof is inspired by Belshaw and Borwein's proof \cite{BB2013} of the non $b$-strong normality of the Chapernowne numbers $\xi_{\B{N},b,1}$. The main idea is to define an increasing sequence (in the index $k$) of positive integers $d_{b,c,k}$, and consider the number of $1$s in the number $\xi_{\B{N}\backslash A,b,c}$ up to the $d_{b,c,k}$-th $b$-ary digit. We will show that there are an excess of $1$s. 

To this end, we start with some properties about the special sequence over which we will eventually take limits.

\begin{lemma}\label{dbckga} Let $c\geqslant 1$ be a real number. For the integers $b\geqslant 2$ and $k\geqslant 1$, let $d_{b,c,k}$ denote the number of $b$-ary digits in the non-decreasing concatenation of the first $2b^{k-1}-1$ positive integers written in the base $b$, where each integer of is repeated $\lfloor c^{\ell_b(a_n)}\rfloor$ times. Then $$d_{b,c,k}=\lfloor c^{k}\rfloor k b^{k-1}+\sum_{n=1}^{k-1}\lfloor c^{n}\rfloor n b^{n-1}(b-1).$$ Moreover, we have $$d_{b,c,k}=\frac{cb+b-2}{b(cb-1)}\cdot k(cb)^{k}\cdot\left(1+O\left(\frac{1}{k}\right)\right)
.$$
\end{lemma}

\begin{proof} The first assertion of the lemma is immediate. 

For the second, note that $\lfloor c^n\rfloor=c^n-\{c^n\}$. Thus $$d_{b,c,k}= c^{k} k b^{k-1}+\sum_{n=1}^{k-1} c^{n} n b^{n-1}(b-1)-\left(\{ c^{k}\} k b^{k-1}+\sum_{n=1}^{k-1}\{ c^{n}\} n b^{n-1}(b-1)\right).$$ 
Now 
\begin{align*}c(b-1)\sum_{n=1}^{k-1}n(cb)^{n-1}&=c(b-1)\cdot \frac{d}{dx} \left.\left\{\sum_{n=0}^{k-1}x^{n}\right\}\right|_{x=cb}\\
&=c(b-1)\cdot \frac{d}{dx} \left.\left\{\frac{x^{k}-1}{x-1}\right\}\right|_{x=cb}\\
& =c(b-1)\frac{k(cb)^{k-1}(cb-1)-((cb)^{k}-1)}{(cb-1)^2} \\
& =c(b-1)\frac{k(cb)^{k}-k(cb)^{k-1}-(cb)^{k}+1}{(cb-1)^2} \\
& =\frac{c(b-1)}{(cb-1)^2}\cdot k(cb)^{k}\cdot\left(1-\frac{1}{cb}-\frac{1}{k}+\frac{1}{k(cb)^k}\right)\\
& =\frac{(b-1)}{b(cb-1)}\cdot k(cb)^{k}\cdot\left(1+O\left(\frac{1}{k}\right)\right),
\end{align*} so also, $$\{ c^{k}\} k b^{k-1}+\sum_{n=1}^{k-1}\{ c^{n}\} n b^{n-1}(b-1)=\delta_cO(kb^k),$$ where $\delta_c=0$ if $c\in\B{N}$ and $1$ otherwise; this accounts for the fact that this term does not appear if $c=1.$

Hence \begin{align*}d_{b,c,k}&=\frac{1}{b}\cdot k(cb)^{k}+\frac{(b-1)}{b(cb-1)}\cdot k(cb)^{k}\cdot\left(1+O\left(\frac{1}{k}\right)\right)+\delta_cO(kb^k)\\
&=\frac{cb+b-2}{b(cb-1)}\cdot k(cb)^{k}\cdot\left(1+O\left(\frac{1}{k}\right)\right).\qedhere
\end{align*}
\end{proof}

As our proof will use the comparison of $m_{1,b}(\xi_{\B{N}\backslash A,b,c};d_{b,c,k})$ with  $m_{1,b}(\xi_{\B{N},b,c};d_{b,c,k})$, we now provide two lemmas. The first provides a way for us to compare these quantities, and the second gives the value of $m_{1,b}(\xi_{\B{N},b,c};d_{b,c,k})$.

\begin{lemma}\label{uptoDbkc} If $A\subseteq\B{N}$ and $A(x)$ is as defined in Proposition \ref{upAc}, then \begin{align*} m_{1,b}(\xi_{\B{N}\backslash A,b,c};d_{b,c,k}) \geqslant m_{1,b}(\xi_{\B{N},b,c};&d_{b,c,k})-\lfloor c^k\rfloor k(A(2b^{k-1}-1)-A(b^{k-1}-1))\\ &-\sum_{n=1}^{k-1}\lfloor c^n\rfloor n(A(b^{n}-1)-A(b^{n-1}-1)).\end{align*}
\end{lemma}

\begin{proof} The number $d_{b,c,k}$ was defined to be the number of $b$-ary digits in the concatenation of the first $2b^{k-1}-1$ positive integers written in the base $b$, where each integer of is repeated $\lfloor c^{\ell_b(a_n)}\rfloor$ times. For $n=1,\ldots, k-1$, there are exactly $A(b^{n}-1)-A(b^{n-1}-1)$ integers in $A$ of length $n$ and there are a total of $A(2b^{k-1}-1)-A(b^{k-1}-1)$ of length $k$. The maximal contribution these numbers could have is if each of the numbers in $A$ had all of it's $b$-ary digits equal to $1$. The inequality in the lemma follows immediately from this observation.
\end{proof}

\begin{lemma}\label{Dchamp} Let $d_{b,c,k}$ be as defined in Lemma \ref{dbckga}. Then for large enough $k$ $$m_{1,b}(\xi_{\B{N},b,c};d_{b,c,k})-\frac{d_{b,c,k}}{b}=(cb)^k\left(\frac{b-1}{b^2}+\frac{1}{b^2(cb-1)}+\delta_c O\left(\frac{1}{c^k}\right)\right),$$ where $\delta_c=0$ if $c\in\B{N}$ and $1$ otherwise.
\end{lemma}  

\begin{proof} Recall that the number $d_{b,c,k}$ denotes the number of $b$-ary digits in the non-decreasing concatenation of the first $2b^{k-1}-1$ positive integers written in the base $b$, where each integer of is repeated $\lfloor c^{\ell_b(a_n)}\rfloor$ times. Note also, that in the set of numbers of length $n$ with any single leading $b$-ary digit, the frequency in the non-leading digits of any specific $b$-ary digit is exactly $1/b$. Thus \begin{align*} m_{1,b}(\xi_{\B{N},b,c};d_{b,c,k})&=\sum_{n=1}^{k-1}\lfloor c^n\rfloor \Bigg(\underset{\text{from length $n$ numbers}}{\underbrace{\frac{(n-1)b^{n-1}(b-1)}{b}}_{\text{the number of non-leading $1$s}}}+\underset{\text{from length $n$ numbers}}{\underbrace{b^{n-1}}_{\text{the number of leading $1$s}}}\ \Bigg)\\
+\lfloor & c^k\rfloor \Bigg(\underset{\text{from numbers in $[b^{k-1},2b^{k-1}-1]$}}{\underbrace{\frac{(k-1)b^{k-1}}{b}}_{\text{the number of non-leading $1$s}}}+\underset{\text{from numbers in $[b^{k-1},2b^{k-1}-1]$}}{\underbrace{b^{k-1}}_{\text{the number of leading $1$s}}} \Bigg)\\
&=\frac{\lfloor c^{k}\rfloor k b^{k-1}+\sum_{n=1}^{k-1}\lfloor c^{n}\rfloor n b^{n-1}(b-1)}{b}\\
&\qquad\qquad-\frac{\lfloor c^{k}\rfloor b^{k-1}+\sum_{n=1}^{k-1}\lfloor c^{n}\rfloor b^{n-1}(b-1)}{b}+\sum_{n=1}^{k}\lfloor c^n\rfloor b^{n-1}\end{align*}
Applying the first assertion of Lemma \ref{dbckga}, we have 
\begin{align*}m_{1,b}(\xi_{\B{N},b,c};d_{b,c,k})-\frac{d_{b,c,k}}{b}&=-\frac{\lfloor c^{k}\rfloor b^{k-1}+\sum_{n=1}^{k-1}\lfloor c^{n}\rfloor b^{n-1}(b-1)}{b}+\sum_{n=1}^{k}\lfloor c^n\rfloor b^{n-1}\\
&=\lfloor c^{k}\rfloor b^{k-1}\left(1-\frac{1}{b}\right)+\frac{1}{b}\sum_{n=1}^{k-1}\lfloor c^{n}\rfloor b^{n-1}\\
&=\lfloor c^{k}\rfloor b^{k-1}\left(1-\frac{1}{b}\right)+\frac{c}{b}\sum_{n=1}^{k-1} (cb)^{n-1}+\frac{1}{b}\sum_{n=1}^{k-1} \{c^{n}\}b^{n-1}\\
&=c^{k}b^{k-1}\left(1-\frac{1}{b}\right)+\frac{c}{b}\left(\frac{(cb)^{k-1}-1}{cb-1}\right)+\delta_c O(b^k)\\
&=(cb)^k\left(\frac{b-1}{b^2}\right)+\frac{1}{b^2}\left(\frac{(cb)^{k}}{cb-1}-\frac{cb}{cb-1}\right)+\delta_c O(b^k)\\
&=(cb)^k\left(\frac{b-1}{b^2}+\frac{1}{b^2(cb-1)}+\delta_c O\left(\frac{1}{c^k}\right)\right),
\end{align*} 
which proves the lemma.
\end{proof}

The next lemma provides a sufficient (but not necessary) condition for the non $b$-simple strong normality of the number $\xi_{\B{N}\backslash A,b,c}$. 

\begin{lemma}\label{Acbk}
Let $c\geqslant 1$ be a real number, $A\subseteq \B{N}$, $d_{b,c,k}$ be as defined in Lemma \ref{dbckga}, and suppose that for large enough $k$, we have $$m_{1,b}(\xi_{\B{N}\backslash A,b,c};d_{b,c,k})-\frac{d_{b,c,k}}{b} \geqslant r (cb)^k$$ for some positive constant $r$. Then $\xi_{\B{N}\backslash A,b,c}$ is not $b$-simply strongly normal. In particular, $\xi_{\B{N}\backslash A,b,c}$ is not $b$-strongly normal.
\end{lemma}

\begin{proof} Using the assumptions of the lemma, we have
$$
\frac{m_{1,b}(\xi_{\B{N}\backslash A,b,c};d_{b,c,k})-{d_{b,c,k}}/{b}}{\sqrt{2d_{b,c,k}\log\log d_{b,c,k}}} \geqslant \frac{r (cb)^{k}}{\sqrt{2d_{b,c,k}\log\log d_{b,c,k}}}.$$ Now using the second assertion of Lemma \ref{dbckga}, for any $\eps>0$ and $k$ large enough $k$, $$\sqrt{2d_{b,c,k}\log\log d_{b,c,k}}\leqslant \left(\frac{d_{b,c,k}}{2}\cdot \frac{b(bc-1)}{cb+b-2}\right)^{1/2+\eps}\leqslant\left(k(cb)^{k}\right)^{1/2+\eps}.$$ Thus $$\lim_{k\to\infty}\frac{m_{1,b}(\xi_{\B{N}\backslash A,b,c};d_{b,c,k})-{d_{b,c,k}}/{b}}{\sqrt{2d_{b,c,k}\log\log d_{b,c,k}}}\geqslant \lim_{k\to\infty}\frac{r(cb)^k}{\left(k(cb)^{k}\right)^{1/2+\eps}}=\infty,$$ whence the bound on the limit supremum in the definition of $b$-simple strong normality cannot hold for the number $\xi_{\B{N}\backslash A,b,c}.$
\end{proof}

In addition to the above results, we will need the following classical result, which is an easy exercise for the curious reader.

\begin{lemma}[Abel Summation] If $a_n,b_n\in\B{C}$ for all $n\in\B{N}$, then for all $k\geqslant 1$ $$\sum_{n=1}^k a_n(b_{n+1}-b_n)=a_{k+1}b_{k+1}-a_1b_1-\sum_{n=1}^kb_{n+1}(a_{n+1}-a_n).$$ 
\end{lemma}

We make use of Abel summation in the following way.

\begin{corollary}\label{abel} If $a_n,b_n\geqslant 0$ for all $n$ and $\{a_n\}_{n\geqslant 0}$ is nondecreasing, then $$\sum_{n=1}^k a_n(b_{n+1}-b_n)\leqslant a_{k+1}b_{k+1}.$$
\end{corollary}

With all of the above preliminaries finished, we are now able to present the proof of Theorem \ref{upAc}.

\begin{proof}[Proof of Theorem \ref{upAc}] Applying Lemmas \ref{uptoDbkc} and \ref{Dchamp} along with Corollary \ref{abel} and the removal of some positive terms, we have \begin{align*} 
m_{1,b}(\xi_{\B{N}\backslash A,b,c};d_{b,c,k}) 
&\geqslant m_{1,b}(\xi_{\B{N},b,c};d_{b,c,k}) - \sum_{n=1}^{k-1} c^n n(A(b^{n}-1)-A(b^{n-1}-1))\\ 
&\qquad\qquad- c^k k(A(2b^{k-1}-1)-A(b^{k-1}-1))\\
&\geqslant \frac{d_{b,c,k}}{b}+(cb)^k\left(\frac{b-1}{b^2}+\frac{1}{b^2(cb-1)}+\delta_c O\left(\frac{1}{c^k}\right)\right)\\
&\qquad- c^k kA(b^{k-1}-1)- c^k k\left(A(2b^{k-1}-1)-A(b^{k-1}-1)\right).
\end{align*}
Thus $$ m_{1,b}(\xi_{\B{N}\backslash A,b,c};d_{b,c,k})-\frac{d_{b,c,k}}{b}\geqslant \frac{(cb)^k}{b^2}\left(b-1+\frac{1}{cb-1}+\delta_c O\left(\frac{1}{c^k}\right)\right)- c^k kA(2b^{k-1}).$$
Supposing that $A(x)\leqslant \ga x/\log x$, we have \begin{align*} m_{1,b}(\xi_{\B{N}\backslash A,b,c};d_{b,c,k})-\frac{d_{b,c,k}}{b}&\geqslant \frac{(cb)^k}{b^2}\left(b-1+\frac{1}{cb-1}+\delta_c O\left(\frac{1}{c^k}\right)\right)\\
&\qquad\qquad- c^k \ga \frac{k}{k-1}\cdot\frac{2b^{k-1}}{\log b}\left(\frac{1}{1+\frac{\log 2}{(k-1)\log b}}\right)\\
&\geqslant \frac{(cb)^k}{b^2}\left(b-1+\frac{1}{cb-1}- \ga \cdot\frac{2b}{\log b}+O\left(\frac{1}{k}\right)\right).
\end{align*} Applying Lemma \ref{Acbk} to this last inequality, we have for any subset $A\subseteq\B{N}$ such that $$\ga<\left(b-1+\frac{1}{cb-1}\right)\frac{\log b}{2b}
=\left(1-\frac{1}{b}+\frac{1}{b(bc-1)}\right)\cdot \frac{\log b}{2},$$ the number $\xi_{\B{N}\backslash A,b,c}$ is not $b$-simply strongly normal. 
\end{proof}

\section{Some open questions and thoughts for the future}

In this paper, we showed that a large class of Copeland-Erd\H{o}s numbers are not strongly normal. The main thrust of our argument was the comparison of our class with the Champernowne (or generalised Champernowne) number. This argument yielded nice results, but there is a lot left to explore here. Indeed, it would be reasonable to suppose that the decimal formed by the sequence of prime numbers is also not strongly normal to the scale of ten, but of this we have no proof.

The following few avenues and questions may lead to some new understanding of both normality and strong normality as well as their relationship.

First, our argument address $\xi_{\B{N}\backslash A, b,c}$ where $A\subseteq \B{N}$ is a sufficiently thin set. For example, if $A$ is the set of primes our method works, but if $A$ is larger than that in any real asymptotic sense, then our argument fails. In contrast to this, the Copeland-Erd\H{o}s normality result holds for a much larger class of numbers. 

\begin{question} Do there exist strongly normal Copeland-Erd\H{o}s numbers?
\end{question}

Second, Davenport and Erd\H{o}s \cite{DE1952} proved that if $f(x)$ is a (non-constant) polynomial in $x$, such that $f(n)\in\B{N}$ for all $n\in\B{N}$, then the real number $$0.(f(1))_b(f(2))_b(f(3))_b\cdots$$ is normal to the base $b$. Their result was later generalised in many ways by Nakai and Shiokawa \cite{NS1990a,NS1990b,NS1992,NS1997}. This begs the following question.

\begin{question} Let $f(x)$ be a (non-constant) polynomial in $x$, such that $f(n)\in\B{N}$ for all $n\in\B{N}$. Can the number $0.(f(1))_b(f(2))_b(f(3))_b\cdots$ be $b$-strongly normal?
\end{question}

Nakai and Shiokawa \cite{NS1997} showed that if $f(x)$ is a (non-constant) polynomial in $x$, such that $f(n)\in\B{N}$ for all $n\in\B{N}$, then the real number $$\ga(f,b):=0.(f(2))_b(f(3))_b(f(5))_b\cdots (f(p))_b\cdots,$$ where $p$ runs through the prime numbers, is $b$-normal. 

\begin{question} Is the number $\ga(f,b)$ strongly normal for any choice of $f$ and $b$?
\end{question}

Finally, the following question, while not specifically about strong normality, is certainly evident given the existing literature, though it has not explicitly been formulated before.

\begin{question} Let $b\geqslant 2$ be an integer and $a_1,a_2,a_3,\ldots$ be an increasing sequence of positive integers such that the number $0.(a_1)_b(a_2)_b(a_3)_b\cdots$ is $b$-normal. If $f(x)$ is a (non-constant) polynomial in $x$, such that $f(n)\in\B{N}$ for all $n\in\B{N}$, then is the number $0.(f(a_1))_b(f(a_2))_b(f(a_3))_b\cdots$ $b$-normal?
\end{question}

\bibliographystyle{amsplain}
\providecommand{\bysame}{\leavevmode\hbox to3em{\hrulefill}\thinspace}
\providecommand{\MR}{\relax\ifhmode\unskip\space\fi MR }
\providecommand{\MRhref}[2]{%
  \href{http://www.ams.org/mathscinet-getitem?mr=#1}{#2}
}
\providecommand{\href}[2]{#2}


\end{document}